\documentclass{amsart}

\usepackage{amsfonts}
\usepackage{amsmath}
\usepackage{hyperref}
\usepackage{amssymb}

\newtheorem{theorem}{Theorem}[section]

\newtheorem{lemma}[theorem]{Lemma}

\newtheorem{corollary}[theorem]{Corollary}
\newtheorem{conjecture}[theorem]{Conjecture}
\newtheorem{cit}[theorem]{Citation}

\newtheorem*{main:BNSR_Houghton}{Theorems~\ref{thrm:pos} and~\ref{thrm:BNSR_Houghton}}

\theoremstyle{definition}
\newtheorem{definition}[theorem]{Definition}
\newtheorem{remark}[theorem]{Remark}

\newcommand{\Z}{\mathbb{Z}}
\newcommand{\N}{\mathbb{N}}

\newcommand{\R}{\mathbb{R}}

\newcommand{\lk}{\operatorname{lk}}

\newcommand{\dlk}{\lk^\downarrow\!}
\newcommand{\alk}{\lk^\uparrow\!}

\newcommand{\defeq}{:=}

\DeclareMathOperator{\Hom}{Hom}
\DeclareMathOperator{\Symm}{Symm}
\DeclareMathOperator{\F}{F}
\DeclareMathOperator{\FP}{FP}
\DeclareMathOperator{\CAT}{CAT}

\DeclareMathOperator{\id}{id}
\DeclareMathOperator{\conv}{conv}
\DeclareMathOperator{\image}{image}

\numberwithin{equation}{section}

\begin{document}

\title{The BNSR-invariants of the Houghton groups, concluded}
\date{\today}


\author{Matthew C.~B.~Zaremsky}
\address{Department of Mathematics and Statistics, University at Albany (SUNY), Albany, NY 12222}
\email{mzaremsky@albany.edu}

\begin{abstract}
We give a complete computation of the BNSR-invariants $\Sigma^m(H_n)$ of the Houghton groups $H_n$. Partial results were previously obtained by the author, with a conjecture about the full picture, which we now confirm. The proof involves covering relevant subcomplexes of an associated $\CAT(0)$ cube complex by their intersections with certain locally convex subcomplexes, and then applying a strong form of the Nerve Lemma. A consequence of the full computation is that for each $1\le m\le n-1$, $H_n$ admits a map onto $\Z$ whose kernel is of type $\F_{m-1}$ but not $\F_m$, and moreover no such kernel is ever of type $\F_{n-1}$.
\end{abstract}

\maketitle
\thispagestyle{empty}

\section*{Introduction}

The Bieri--Neumann--Strebel--Renz (BNSR) invariants $\Sigma^m(G)$ ($m\in\N$) of a group $G$ are a sequence of geometric invariants, introduced in \cite{bieri87,bieri88}, that reveal a breadth of information about certain subgroups of $G$. They are notoriously difficult to compute, and a full computation has been done for only a handful of relevant groups. Most prominently, the BNSR-invariants are fully computed for all right-angled Artin groups \cite{meier98,bux99}. In \cite{zaremsky17} we computed the BNSR-invariants of the generalized Thompson groups $F_{n,\infty}$ and obtained partial results for the Houghton groups $H_n$. In this paper we finish the computation of the BNSR-invariants of the Houghton groups. Our main result is as follows (with the notation $m(\chi)$ explained in Section~\ref{sec:houghton_chars}).

\begin{main:BNSR_Houghton}
For any $0\ne \chi\in \Hom(H_n,\R)$ we have $[\chi]\in \Sigma^{m(\chi)-1}(H_n)\setminus\Sigma^{m(\chi)}(H_n)$.
\end{main:BNSR_Houghton}

The statement $[\chi]\in \Sigma^{m(\chi)-1}(H_n)$ was proved in \cite{zaremsky17}, and is cited as Theorem~\ref{thrm:pos} here; the new result is Theorem~\ref{thrm:BNSR_Houghton}, that $[\chi]\not\in \Sigma^{m(\chi)}(H_n)$. Since we always have $\Sigma^m(G)\subseteq \Sigma^{m-1}(G)$ for all $m$ and $G$, Theorems~\ref{thrm:pos} and~\ref{thrm:BNSR_Houghton} provide a complete computation of the BNSR-invariants of the Houghton groups.

The proof makes use of the $\CAT(0)$ structure of a natural cube complex on which $H_n$ acts. We exhibit a family of combinatorially convex subcomplexes called ``blankets'' (Definition~\ref{def:blanket}) that cover a relevant $\chi$-superlevel complex, and use a strong form of the Nerve Lemma to conclude that the $\chi$-superlevel complex is not $(m(\chi)-1)$-connected. This then leads reasonably quickly to the result. The geometry of the complex is crucial for getting the regions covered by the blankets to be highly connected enough to apply the Strong Nerve Lemma. We believe that this covering approach could be useful in the future for understanding BNSR-invariants of other groups acting naturally on $\CAT(0)$ cube complexes, and possibly for approaching the $\Sigma^m$-Conjecture (see Conjecture~\ref{conj:sigmam}) for (certain classes of) metabelian groups.

This paper is organized as follows. In Section~\ref{sec:background} we recall some general background. In Section~\ref{sec:houghton_chars} we recall some specific background from \cite{zaremsky17} and state our main result, Theorem~\ref{thrm:BNSR_Houghton}. In Section~\ref{sec:cpx} we recall the $\CAT(0)$ cube complex $X_n$ on which $H_n$ acts, and introduce an important family of subcomplexes called ``blankets'' in $X_n$. Finally in Section~\ref{sec:proof} we prove Theorem~\ref{thrm:BNSR_Houghton}.

\subsection*{Acknowledgments} I am grateful to the anonymous referee for a number of helpful suggestions.

\section{Background}\label{sec:background}

In this section we collect some background on Houghton groups, Morse theory and BNSR-invariants.

\subsection{Houghton groups}\label{sec:houghton}

Let $[n]\defeq \{1,\dots,n\}$. The \emph{Houghton group} $H_n$, introduced in \cite{houghton79}, is the subgroup of $\Symm([n]\times\N)$ consisting of those elements $\eta$ such that for each $1\le i\le n$ there exists $m_i\in\Z$ and $N_i\in\N$ such that $(i,x)\eta=(i,x+m_i)$ for all $x\ge N_i$ (we will always write elements of $\Symm([n]\times\N)$ to the right of their arguments, to sync with the notation in \cite{zaremsky17}). Intuitively such an $\eta$ acts as an ``eventual translation'' on each $\{i\}\times\N$. It is known that $H_n$ is of type $\F_{n-1}$ but not $\F_n$ \cite[Theorem~5.1]{brown87}. Higher dimensional versions of the Houghton groups, due to Bieri and Sach \cite{bieri16,sach16}, have also been developed.

\subsection{BNSR-invariants}\label{sec:bnsr}

A group is of \emph{type $F_m$} if it admits a proper cocompact action on an $(m-1)$-connected CW-complex. Given a group $G$ call a homomorphism $\chi\colon G\to \R$ a \emph{character}, and call two characters \emph{equivalent} if they are positive scalar multiples of each other. The equivalence classes $[\chi]$ of non-trivial characters form the \emph{character sphere} $\Sigma(G)$. The BNSR-invariants $\Sigma^m(G)$ ($m\in\N$) of a group $G$ are certain subspaces of $\Sigma(G)$, defined whenever $G$ is of type $\F_m$. They were introduced for $m=1$ in \cite{bieri87} and $m\ge 2$ in \cite{bieri88}. The definition is as follows.

\begin{definition}[BNSR-invariant]\label{def:BNSR}
Let $G$ be a group acting properly cocompactly on an $(m-1)$-connected CW-complex $Y$, so in particular $G$ is of type $\F_m$. For any $0\ne \chi\in \Hom(G,\R)$ there exists a \emph{character height function} $h_\chi \colon Y \to \R$, that is, a map satisfying $h_\chi(g.y)=\chi(g)+h_\chi(y)$ for all $g\in G$ and $y\in Y$. The $m$th \emph{Bieri--Neumann--Strebel--Renz (BNSR) invariant} $\Sigma^m(G)$ is
\[
\Sigma^m(G) \defeq \{[\chi]\in\Sigma(G)\mid (Y^{t\le h_\chi})_{t\in\R} \text{ is essentially } (m-1)\text{-connected}\}\text{.}
\]
\end{definition}

Here $Y^{t\le h_\chi}$ is the full subcomplex of $Y$ supported on those $v\in Y^{(0)}$ with $t\le h_\chi(v)$. Recall that $(Y^{t\le h_\chi})_{t\in\R}$ being \emph{essentially $(m-1)$-connected} means that for all $t\in\R$ there exists $s\le t$ such that the inclusion $Y^{t\le h_\chi}\to Y^{s\le h_\chi}$ induces the trivial map in $\pi_k$ for all $k\le m-1$.

As is standard, we will write $\Sigma^m(G)^c$ for the complement $\Sigma(G)\setminus\Sigma^m(G)$. Note that $\Sigma^1(G)\supseteq \Sigma^2(G)\supseteq\cdots$. There are also homological BNSR-invariants $\Sigma^m(G;\Z)$, analogous to the homological finiteness properties $\FP_m$, though we will not discuss the homological case much here.

\begin{remark}[Erratum to \cite{zaremsky17,bux04}]
Definition~\ref{def:BNSR} is almost identical to \cite[Definition~1.1]{zaremsky17}, except there the condition on stabilizers was that each $p$-cell stabilizer be of type $\F_{m-p}$ and here we just assume a proper action (so finite stabilizers). The reason for this change is that in order for $h_\chi$ to exist, the cell stabilizers need to lie in $\ker(\chi)$, and assuming a proper action is an easy way to assure this. This condition on $\chi$ killing the stabilizers was accidentally omitted in \cite[Definition~1.1]{zaremsky17} (and in \cite[Definition~8.1]{bux04}, which \cite[Definition~1.1]{zaremsky17} followed), though since in practice it was always applied to situations with finite stabilizers, this error was irrelevant. Also note that one can always choose $Y$ so that the action of $G$ is proper (even free) if one wants, it is just sometimes convenient (in other situations than our present one) to deal with spaces with infinite stabilizers.
\end{remark}

The BNSR-invariants of $G$ form a sort of catalog describing the precise finiteness properties of subgroups of $G$ containing the commutator subgroup $[G,G]$, namely:

\begin{cit}\cite[Theorem~1.1]{bieri10}\label{cit:bnsr_fin_props}
Let $G$ be a group of type $\F_m$ and let $[G,G]\le H\le G$. Then $H$ is of type $\F_m$ if and only if for every $[\chi]\in\Sigma(G)$ such that $\chi(H)=0$ we have $[\chi]\in\Sigma^m(G)$.
\end{cit}

In our computation of $\Sigma^m(H_n)$ for the Houghton groups $H_n$, it turns out the structure is such that the invariant $\Sigma^1(H_n)$ already determines all the $\Sigma^m(H_n)$ in a natural, ``polyhedral'' way. We describe this phenomenon in the following definition.

\begin{definition}[Bieri $\Sigma$-property]\label{def:bieri}
For a finitely generated group $G$ and a subset $S$ of $\Sigma(G)$, denote by $\conv_{\le m}S$ the union of convex hulls in $\Sigma(G)$ of all subsets of at most $m$ non-antipodal elements in $S$. Now suppose $G$ is of type $\F_m$ and suppose that
\[
\Sigma^m(G)^c = \conv_{\le m}\Sigma^1(G)^c \text{.}
\]
Then we say $G$ has the \emph{Bieri $\Sigma^m$-property}. If $G$ has the Bieri $\Sigma^m$-property for all $m$ such that $G$ is of type $\F_m$ then we say $G$ has the \emph{Bieri $\Sigma$-property}.
\end{definition}

Every finitely generated group trivially has the Bieri $\Sigma^1$-property, but finitely presented groups need not have the Bieri $\Sigma^2$-property, in fact Kochloukova found a solvable (even nilpotent-by-abelian) counterexample \cite[Theorem~B]{kochloukova02}. The well known $\Sigma^m$-Conjecture, which seems to originally be due to Bieri though perhaps was first stated in this form by Meinert, can be phrased (in homotopical form) as:

\begin{conjecture}[$\Sigma^m$-Conjecture]\label{conj:sigmam}
Every metabelian group of type $\F_m$ has the Bieri $\Sigma^m$-property.
\end{conjecture}

One can state the $\Sigma^m$-Conjecture simultaneously for all $m$ as the ``$\Sigma$-Conjecture'' that every metabelian group has the Bieri $\Sigma$-property. As seen in Corollary~\ref{cor:houghton_bieri}, our results here imply that the Houghton groups have the Bieri $\Sigma$-property. Meinert proved that the $\Sigma$-Conjecture holds for metabelian groups of finite Pr\"ufer rank \cite{meinert96} and Harlander--Kochloukova proved the $\Sigma^2$-Conjecture in general \cite{harlander04}. In addition to the Houghton groups, some other non-metabelian groups whose BNSR-invariants are fully computed revealing that they have the Bieri $\Sigma$-property include Thompson's group $F$ \cite{bieri10}, its relatives $F_{n,\infty}$ \cite{kochloukova12,zaremsky17} and its braided version \cite{zaremsky18}. The most prominent family of groups whose BNSR-invariants are fully computed is the family of right-angled Artin groups \cite{meier98,bux99}, and from the computation one can see that generally speaking ``most'' right-angled Artin groups do not have the Bieri $\Sigma$-property. This is essentially because it is easy for a collection of subsets of vertices of a flag complex to individually induce disconnected subcomplexes but have their union induce a highly connected (e.g, contractible) subcomplex. On the other hand free groups and free abelian groups do have the Bieri $\Sigma$-property for trivial reasons. We leave a precise classification of which right-angled Artin groups have the Bieri $\Sigma$-property as an exercise for the reader.

\subsection{Morse theory}\label{sec:morse}

We will use the definition of Morse function and the statement of the Morse Lemma from \cite{witzel15,zaremsky17}. We recall these now, and see \cite{witzel15,zaremsky17} for any details we leave out here. Fix an affine cell complex $Y$ in this subsection.

\begin{definition}[Morse function]\label{def:morse}
Let $(h,s)\colon Y \to \R\times\R$ a map such that $h$ and $s$ restrict to affine functions on cells. We call $(h,s)$ a \emph{Morse function} provided that $|s(Y^{(0)})|<\infty$ and there exists $\varepsilon>0$ such that whenever $v$ and $w$ are adjacent $0$-cells in $Y$, either $|h(v)-h(w)|\ge\varepsilon$, or else $h(v)=h(w)$ and $s(v)\ne s(w)$.
\end{definition}

We view $(h,s)$ as a height function via the lexicographic order on $\R\times\R$, and the conditions ensure that adjacent $0$-cells have different heights. Along a given cell, $(h,s)$ achieves its minimum and maximum values at unique $0$-faces. If a $0$-cell $v$ is a $0$-face of a cell at which $(h,s)$ achieves its minimum (maximum) on that cell, that cell belongs to the \emph{ascending (descending) star} of $v$. The \emph{ascending (descending) link} of $v$, denoted $\alk v$ ($\dlk v$) is the link of $v$ in its ascending (descending) star. Write $Y^{p\le h\le q}$ for the full subcomplex of $Y$ supported on those $0$-cells $v$ with $p\le h(v)\le q$.

The following is essentially \cite[Lemma~1.4]{zaremsky17}, phrased slightly differently.

\begin{lemma}[Morse Lemma]\label{lem:morse}
Suppose that for each $0$-cell $v$ with $p\le h(v)<q$ ($q<h(v)\le r$) the ascending (descending) link $\alk v$ ($\dlk v$) is $(m-1)$-connected. Then the inclusion $Y^{q\le h} \to Y^{p\le h}$ ($Y^{h\le q} \to Y^{h\le r}$) induces an isomorphism in $\pi_k$ for each $k\le m-1$ and a surjection in $\pi_m$.
\end{lemma}

As a remark, if $s$ is constant and $h(Y^{(0)})$ is discrete in $\R$, then this definition of Morse function reduces to the original one introduced by Bestvina--Brady in \cite{bestvina97}.

We will need one more topological tool, namely the following strong version of the classical Nerve Lemma.

\begin{cit}\cite[Proposition~1.21]{witzel16}\label{cit:strong_nerve}
Let $X$ be a CW-complex covered by subcomplexes $(X_i)_{i\in I}$ and let $L$ be the nerve of the cover. Let $n\ge 1$. Suppose that any non-empty intersection $X_{i_1}\cap\cdots\cap X_{i_r}$ for $1\le r\le n$ is $(n-r)$-connected. Then $H_k(X)\cong H_k(L)$ for all $k\le n-1$, and $H_n(X)$ surjects onto $H_n(L)$.
\end{cit}

This was not phrased exactly this way in \cite[Proposition~1.21]{witzel16}, but it is straightforward to see this is an equivalent formulation. Note that being not $n$-acyclic implies being not $n$-connected.

\section{Characters of Houghton groups and statement of results}\label{sec:houghton_chars}

In \cite{zaremsky17} a partial computation of $\Sigma^m(H_n)$ was obtained. Before stating the result we need some notation and background from \cite{zaremsky17} regarding characters of $H_n$. For each $1\le i\le n$ the function $\eta \mapsto m_i$, with $m_i$ as in the definition of the Houghton groups, defines an epimorphism $\chi_i \colon H_n \to \Z$, and $\chi_1,\dots,\chi_n$ span $\Hom(H_n,\R)$. Since elements of $H_n$ are bijective we have $m_1+\cdots+m_n=0$ for each $\eta$ and hence $\chi_1+\cdots+\chi_n=0$. In fact $\Hom(H_n,\R)\cong \R^{n-1}$ with basis $\{\chi_1,\dots,\chi_{n-1}\}$, so $\Sigma(H_n)\cong S^{n-2}$. This also implies that for an arbitrary character $\chi$ of $H_n$, written as $\chi=a_1\chi_1+\cdots+a_n\chi_n$ for $a_i\in\R$, the coefficients $(a_1,\dots,a_n)$ are uniquely determined up to shifting by constants $(a,\dots,a)$. The number
\[
m(\chi) \defeq |\{i\mid a_i<\max\{a_1,\dots,a_n\}\}|
\]
is therefore well defined. This measurement will turn out to determine which BNSR-invariants contain $[\chi]$.

It is sometimes convenient to express characters in a ``standard form'' with respect to the characters $\chi_i$.

\begin{definition}[(Ascending) standard form]
If $\chi=a_1\chi_1+\cdots+a_n\chi_n$ we call this expression for $\chi$ a \emph{standard form} if $\max\{a_i\}_{i=1}^n=0$. We call it an \emph{ascending standard form} if $a_1\le\cdots\le a_n=0$.
\end{definition}

Up to shifting by $\chi_1+\cdots+\chi_n=0$ any $\chi$ can be put in standard form. Up to automorphisms of $H_n$, every $\chi$ is equivalent to one in ascending standard form. In particular when trying to determine which BNSR-invariants contain a given character class, without loss of generality it can be expressed in ascending standard form. For $\chi$ in ascending standard form, $m(\chi)$ equals the largest $i$ such that $a_i\ne 0$.

\subsection{Statement of results}\label{sec:results}

The partial results obtained in \cite{zaremsky17} are:

\begin{theorem}\cite{zaremsky17}\label{thrm:pos}
For any $0\ne\chi\in \Hom(H_n,\R)$ we have $[\chi]\in \Sigma^{m(\chi)-1}(H_n)$.
\end{theorem}

It was conjectured in \cite{zaremsky17} that moreover $[\chi]\not\in \Sigma^{m(\chi)}(H_n)$, and this is our main result here:

\begin{theorem}\label{thrm:BNSR_Houghton}
For any $0\ne\chi\in \Hom(H_n,\R)$ we have $[\chi]\not\in \Sigma^{m(\chi)}(H_n)$.
\end{theorem}

We will prove Theorem~\ref{thrm:BNSR_Houghton} in Section~\ref{sec:proof}.

\subsection{Consequences}\label{sec:consequences}

In this subsection we collect some easy consequences of the computation of the $\Sigma^m(H_n)$. First we have some results on finiteness properties of kernels of characters.

\begin{corollary}\label{cor:kernel}
For any $\chi \colon H_n \twoheadrightarrow \Z$ with $m=\min\{m(\chi),m(-\chi)\}$, the kernel of $\chi$ is of type $\F_{m-1}$ but not $\F_m$.
\end{corollary}

\begin{proof}
This is immediate from the computation of the $\Sigma^m(H_n)$ together with Citation~\ref{cit:bnsr_fin_props}.
\end{proof}

\begin{corollary}\label{cor:kernels}
For each $1\le m\le n-1$, $H_n$ admits a map to $\Z$ whose kernel is of type $\F_{m-1}$ but not $\F_m$.
\end{corollary}

\begin{proof}
Choose any $\chi \colon H_n \twoheadrightarrow \Z$ with $m=\min\{m(\chi),m(-\chi)\}$, for example take $\chi=\chi_1+\cdots+\chi_{m(\chi)-1}+2\chi_m$, so $m(\chi)=n-1$ and $m(-\chi)=m$, and then Corollary~\ref{cor:kernel} gives the result.
\end{proof}

We can also conclude the following result about arbitrary normal subgroups.

\begin{corollary}\label{cor:fin_index}
A non-trivial normal subgroup of $H_n$ is of type $\F_{n-1}$ if and only if it has finite index in $H_n$.
\end{corollary}

\begin{proof}
The thing to prove is that if $N$ is a non-trivial normal subgroup of $H_n$ with infinite index, then $N$ is not of type $\F_{n-1}$. As explained at the end of \cite{zaremsky17}, $N$ must contain the second derived subgroup of $H_n$, which is the finite-support alternating group on $[n]\times\N$. Up to replacing $N$ with a finite index supergroup we can assume it contains the derived subgroup of $H_n$, which is the finite-support symmetric group on $[n]\times\N$. Being an infinite index subgroup of $H_n$ containing the commutator subgroup, we see that $N$ lies in the kernel of a character $\chi\colon H_n \twoheadrightarrow \Z$, and since $[\chi]\not\in \Sigma^{n-1}(H_n)$ (since the computation shows $\Sigma^{n-1}(H_n)=\emptyset$), Citation~\ref{cit:bnsr_fin_props} says $N$ is not of type $\F_{n-1}$.
\end{proof}

Another consequence of our computation is that the Houghton groups all have the Bieri $\Sigma$-property (Definition~\ref{def:bieri}).

\begin{corollary}\label{cor:houghton_bieri}
Every $H_n$ has the Bieri $\Sigma$-property.
\end{corollary}

\begin{proof}
We know that $\Sigma^1(H_n)^c=\{[-\chi_1],\dots,[-\chi_n]\}$ (this was already known in \cite[Proposition~8.3]{brown87}). Hence thanks to standard forms, $\conv_{\le m}\Sigma^1(H_n)^c$ equals the set of $[\chi]$ with $m(\chi)\le m$. By Theorems~\ref{thrm:pos} and~\ref{thrm:BNSR_Houghton}, it is also the case that $[\chi]\in \Sigma^m(H_n)^c$ if and only if $m(\chi)\le m$.
\end{proof}

From the computation it is also clear that we can triangulate $\Sigma(H_n)$ into the boundary of an $(n-1)$-simplex in such a way that $\Sigma^m(H_n)^c$ is precisely the $(m-1)$-skeleton of this simplex. This is an example of the ``polyhedral'' behavior we indicated before defining the Bieri $\Sigma$-property.

\section{Complexes}\label{sec:cpx}

In this section we recall the $\CAT(0)$ cube complex on which $H_n$ acts, and define an important family of $\CAT(0)$ subcomplexes called blankets (Definition~\ref{def:blanket}).

\subsection{Cube complex}\label{sec:cube}

There is a natural cube complex $X_n$ on which $H_n$ acts, which we recall now. Everything in this subsection is taken from \cite{zaremsky17}. First define $M_n$ to be the monoid of injections $\phi\colon [n]\times\N \to [n]\times\N$ that are eventual translations (i.e., satisfy the same condition as elements of $H_n$ except they need not be surjective), so $H_n$ consists precisely of the bijective elements of $M_n$. The $0$-skeleton of $X_n$ is defined to be $M_n$.

To define the $1$-skeleton of $X_n$ we need to recall some important elements $t_i$ of $M_n$. For each $1\le i\le n$ define $t_i\in M_n$ to be
\[
t_i \colon (j,x) \mapsto \left\{\begin{array}{ll} (j,x) &\text{ if } j\ne i \\ (j,x+1) &\text{ if } j=i \end{array}\right.\text{.}
\]
Now declare that two $0$-cubes $\phi,\psi$ in $X_n$ (i.e., elements of $M_n$) span a $1$-cube whenever $\phi=t_i\circ \psi$ or $\psi=t_i\circ \phi$ for some $1\le i\le n$. Such a $1$-cube is \emph{labeled by $t_i$}. (Since our maps act from the right $t_i\circ \phi$ means \emph{pre}compose with $t_i$.) We define the higher dimensional cubes of $X_n$ by declaring that for every $\phi\in M_n$ and every $K\subseteq [n]$ there is a $|K|$-cube spanned by
\[
\left\{\left.\left(\prod_{i\in I}t_i\right)\circ\phi \right| I\subseteq K\right\}\text{.}
\]
Note that the $t_i$ all commute with each other, so specifying an order in the product is unnecessary. Since more subscripts and superscripts will soon appear, we will now write $X$ for $X_n$, and there should be no risk of ambiguity.

It is known that $X$ is a $\CAT(0)$ cube complex. The group $H_n$ acts on $M_n$ from the right via $(\phi)\eta \defeq \phi \circ \eta$, and this extends to an action of $H_n$ on $X$. Each cube stabilizer is finite. There is an $H_n$-invariant Morse function $f\colon X \to \R$ defined on $X^{(0)}$ by
\[
f(\phi)\defeq |([n]\times\N)\setminus \image(\phi)|\text{.}
\]
Each sublevel set $X^{f\le q}$ is $H_n$-cocompact. Note that $X^{f\le 0}$ (that is to say $X^{f=0}$) is precisely $H_n$.

Since elements of $M_n$ are eventual translations just like elements of $H_n$, any character $\chi\colon H_n\to \R$ naturally extends to a monoid homomorphism $\chi\colon M_n\to \R$ given by the same definition as on $H_n$. Then viewing $M_n$ as $X^{(0)}$, any $\chi$ extends to a continuous map $\chi \colon X\to\R$. The lexicographically ordered function $(\chi,f)$ is a Morse function in the sense of Definition~\ref{def:morse} on any $X^{f\le q}$.

\subsection{Blankets}\label{sec:blankets}

Blankets are certain subcomplexes of $X$ that we will use later to cover the complex $X^{0\le \chi}$. The definition is as follows.

\begin{definition}[Blanket]\label{def:blanket}
For $K\subseteq [n]$ consider the subcomplex $\displaystyle\bigcap\limits_{i\in K}X^{\chi_i \le 0}$ of $X$. We will call any connected component of such a subcomplex a \emph{$K$-blanket}, and generally refer to $K$-blankets for arbitrary $K$ as \emph{blankets}.
\end{definition}

Recall that a subcomplex $Z$ of a $\CAT(0)$ cube complex $Y$ is \emph{locally combinatorially convex} if every link in $Z$ of a $0$-cube $z\in Z$ is a full subcomplex of the link of $z$ in $Y$, and \emph{combinatorially convex} if it is connected and locally combinatorially convex. It is well known (for example see \cite[Lemma~2.12]{haglund08}) that combinatorially convex subcomplexes are themselves $\CAT(0)$, hence contractible. In particular each connected component of a locally combinatorially convex subcomplex is contractible.

\begin{lemma}[Blankets are $\CAT(0)$]\label{lem:blankets_cat0}
For any $K$, $\displaystyle\bigcap\limits_{i\in K}X^{\chi_i \le 0}$ is locally combinatorially convex. In particular, blankets are combinatorially convex in $X$, hence $\CAT(0)$ and contractible.
\end{lemma}

\begin{proof}
It is enough to show that each $X^{\chi_i \le 0}$ is locally combinatorially convex. Note that if $\phi,\psi$ are adjacent $0$-cubes in $X$, say with $\psi=t_j\circ \phi$, then $\chi_i(\psi)-\chi_i(\phi)$ is $0$ if $i\ne j$ and $1$ if $i=j$. Thus if we have a cube $C$ containing $\phi$ and $\psi_1,\dots,\psi_r$ are the $0$-faces of $C$ adjacent to $\phi$, then the maximum and minimum values of $\chi_i$ on $C$ lie in $\{\chi_i(\phi),\chi_i(\psi_1),\dots,\chi_i(\psi_r)\}$. In particular if $\phi\in X^{\chi_i \le 0}$ and these $\psi_i$ lie in the link of $\phi$ in $X^{\chi_i \le 0}$, then all of $C$ lies in $X^{\chi_i \le 0}$. This shows that the link of $\phi$ in $X^{\chi_i \le 0}$ is a full subcomplex of the link of $\phi$ in $X$.
\end{proof}

\begin{corollary}[Intersections of blankets are blankets]\label{cor:sheet_int}
Let $Z_1,\dots,Z_r$ be blankets, say with $Z_i$ a $K_i$-blanket. Then if $Z_1\cap\cdots\cap Z_r$ is non-empty it is contractible, and in fact is a $(K_1\cup\cdots\cup K_r)$-blanket.
\end{corollary}

\begin{proof}
Since each $Z_i$ is combinatorially convex, any non-empty $Z_1\cap\cdots\cap Z_r$ is combinatorially convex, hence contractible. Moreover, as a (contractible hence) connected subcomplex of $\displaystyle\bigcap\limits_{i\in K_1\cup\cdots\cup K_r}X^{\chi_i \le 0}$ we know it lies in a $(K_1\cup\cdots\cup K_r)$-blanket. It also contains a $(K_1\cup\cdots\cup K_r)$-blanket for trivial (general topological) reasons, and hence it must equal a $(K_1\cup\cdots\cup K_r)$-blanket.
\end{proof}

\section{The proof}\label{sec:proof}

In this section we will use blankets and the Morse function $(\chi,f)$ to prove our main result, Theorem~\ref{thrm:BNSR_Houghton}. Without loss of generality $\chi=a_1\chi_1+\cdots+a_n\chi_n$ is in ascending standard form, so $a_1\le\cdots\le a_n=0$. Let us write $X_{f\le k}$ for $X^{f\le k}$ and $X_{f\le k}^{t\le\chi}$ for $X_{f\le k}\cap X^{t\le \chi}$. As a remark, in what follows we may occasionally implicitly assume $n\ge 2$; the only character of $H_1$ is $0$, so while our main results are (vacuously) true for $n=1$, some of the arguments used in this section may not literally be true for $n=1$.

\begin{lemma}\label{lem:drop_essential}
If $X_{f\le 3n-3}^{0\le\chi}$ is not $(m(\chi)-1)$-connected then $[\chi]\in \Sigma^{m(\chi)}(H_n)^c$.
\end{lemma}

\begin{proof}
We proceed by contrapositive. If $[\chi]\in \Sigma^{m(\chi)}(H_n)$, then since $X_{f\le 3n-3}$ is $H_n$-cocompact we know the filtration $(X_{f\le 3n-3}^{t\le\chi})_{t\in\R}$ is essentially $(m(\chi)-1)$-connected. By \cite[Proposition~6.6]{zaremsky17} every ascending link with respect to $(\chi,f)$ of a $0$-cube in $X_{f\le 3n-3}$ is $(m(\chi)-2)$-connected, so by Lemma~\ref{lem:morse} the inclusion $X_{f\le 3n-3}^{t\le\chi} \to X_{f\le 3n-3}^{s\le\chi}$ (for any $s\le t$) induces an isomorphism in $\pi_k$ for all $k\le m(\chi)-2$ and a surjection in $\pi_{m(\chi)-1}$. We are assuming that for all $t$ there exists $s\le t$ such that this inclusion induces the trivial map in $\pi_k$ for all $k\le m(\chi)-1$, so in fact for such $s$ the complex $X_{f\le 3n-3}^{s\le\chi}$ is $(m(\chi)-1)$-connected. Without loss of generality $s\in \chi(H_n)$ and so after translating by an element of $H_n$ we get $X_{f\le 3n-3}^{s\le\chi} \cong X_{f\le 3n-3}^{0\le\chi}$, so $X_{f\le 3n-3}^{0\le\chi}$ is $(m(\chi)-1)$-connected.
\end{proof}

Our goal now is to prove that $X_{f\le 3n-3}^{0\le\chi}$ is not $(m(\chi)-1)$-connected. We will cover it with its intersection with certain blankets and apply the Strong Nerve Lemma. For each $1\le i\le n$ let $(Z_i^\alpha)_{\alpha\in I_i}$ be the collection of $\{i\}$-blankets in $X$ (so the connected components of $X^{\chi_i\le 0}$), and set
\[
Y_i^\alpha \defeq Z_i^\alpha \cap X_{f\le 3n-3}^{0\le\chi} \text{.}
\]
Here $I_i$ is just an appropriate indexing set.

\begin{lemma}\label{lem:cover}
The $Y_i^\alpha$ for $1\le i\le m(\chi)$ cover $X_{f\le 3n-3}^{0\le\chi}$.
\end{lemma}

\begin{proof}
It suffices to show that
\[
X^{0\le\chi}\subseteq \bigcup\limits_{i=1}^{m(\chi)} X^{\chi_i\le 0} \text{.}
\]
First note that since $\chi=a_1\chi_1+\cdots+a_{m(\chi)}\chi_{m(\chi)}$ with $a_i<0$ for all $i$, any $0$-cube $v$ in $X$ satisfying $\chi(v)\ge 0$ must satisfy $\chi_i(v)\le 0$ for some $i$. This proves that the inclusion is true on the $0$-skeleton. Now take an arbitrary cube in $X^{0\le\chi}$ and let $v$ be its $0$-face at which $f$ is maximized. Then all $0$-faces $w$ of the cube satisfy $\chi_i(w)\le \chi_i(v)$, and hence as soon as $v$ lies in $X^{\chi_i\le 0}$ so does the cube.
\end{proof}

\begin{lemma}\label{lem:good_pieces}
Any non-empty intersection of any number of subcomplexes of the form $Y_i^\alpha$ (with $1\le i\le m(\chi)$) is $(m(\chi)-2)$-connected.
\end{lemma}

\begin{proof}
For such an intersection to be non-empty, it must feature at most one term of the form $Y_i^\alpha$ for each $i$, so without loss of generality it is $Y_{i_1}^{\alpha_1} \cap\cdots\cap Y_{i_r}^{\alpha_r}$, with the $i_j$ pairwise distinct (here $\alpha_j\in I_{i_j}$). Call this intersection $Y$, and let $Z$ be $Z_{i_1}^{\alpha_1} \cap\cdots\cap Z_{i_r}^{\alpha_r}$, so $Y=Z\cap X_{f\le 3n-3}^{0\le \chi}$. To understand $Y$ we will now apply Morse theoretic techniques to $Z$, similar to those applied to $X$ in \cite{zaremsky17}. The first step is to get from $Z$ to $Z_{f\le 3n-3}$. By Corollary~\ref{cor:sheet_int}, $Z$ is contractible. If $\phi$ and $\psi$ are $0$-cubes of $X$ with $\phi=t_i\circ\psi$ and $\phi\in Z$ then $\psi\in Z$. Hence for any $0$-cube $\phi$ in $Z$, the $f$-descending link of $\phi$ in $X$ lies in $Z$. Since this is $(n-2)$-connected for $f(\phi)>2n-1$ (\cite[Citation~6.4]{zaremsky17}, \cite[Lemma~3.52]{lee12}), we see that $Z_{f\le 3n-3}$ is $(n-2)$-connected, hence $(m(\chi)-2)$-connected (here we are just using $f$ as a standard Morse function as in \cite{bestvina97}). Now we need to get from $Z_{f\le 3n-3}$ to $Y=Z_{f\le 3n-3}^{0\le \chi}$, and to do this we will use the Morse function (in the sense of Definition~\ref{def:morse}) $(\chi,f)$. For $\phi$ a $0$-cube in $Z_{f\le 3n-3}$, as in \cite{zaremsky17} the $(\chi,f)$-ascending link of $\phi$ in $X_{f\le 3n-3}$ is the join of an $f$-ascending part and an $f$-descending part. The $f$-descending part lies in $Z$ for the same reasons as above. The $f$-ascending part lies in $Z$ since it consists of directions labeled by $t_k$ for $m(\chi)<k\le n$, and each $\chi_{i_j}$ is constant in those directions. Hence the ascending link of $\phi$ in $Z_{f\le 3n-3}$ equals the ascending link of $\phi$ in $X_{f\le 3n-3}$, and this is $(m(\chi)-2)$-connected by \cite[Proposition~6.6]{zaremsky17}. Now the Morse Lemma~\ref{lem:morse} tells us that $Y=Z_{f\le 3n-3}^{0\le \chi}$ is $(m(\chi)-2)$-connected.
\end{proof}

Let $L$ be the nerve of the covering of $X_{f\le 3n-3}^{0\le\chi}$ by the $Y_i^\alpha$. Since $[\chi]\in \Sigma^{m(\chi)-1}(H_n)$ by Theorem~\ref{thrm:pos}, we know from Lemma~\ref{lem:drop_essential} that $X_{f\le 3n-3}^{0\le\chi}$ is $(m(\chi)-2)$-connected, so the Strong Nerve Lemma (Citation~\ref{cit:strong_nerve}), which applies by Lemma~\ref{lem:good_pieces}, says that $L$ is also $(m(\chi)-2)$-connected. To prove Theorem~\ref{thrm:BNSR_Houghton}, the last result we need is that $L$ is not $(m(\chi)-1)$-acyclic.

\begin{lemma}\label{lem:houghton_nerve}
The nerve $L$ is not $(m(\chi)-1)$-acyclic.
\end{lemma}

\begin{proof}
Each vertex of $L$ has a \emph{type} in $[m(\chi)]$, given by declaring that the vertex corresponding to $Y_i^\alpha$ has type $i$. Vertices of the same type cannot be adjacent, so $L$ is $(m(\chi)-1)$-dimensional. Thus it suffices to exhibit a non-trivial $(m(\chi)-1)$-cycle. For this we will find, for each $1\le i\le m(\chi)$, a pair of distinct vertices $Y_i^1$ and $Y_i^2$ of type $i$, such that $Y_1^{\epsilon_1}\cap\cdots\cap Y_{m(\chi)}^{\epsilon_{m(\chi)}} \ne\emptyset$ for every choice of $\epsilon_j\in\{1,2\}$. This will then yield an embedded $(m(\chi)-1)$-sphere in $L$, which must be homologically non-trivial since $L$ is $(m(\chi)-1)$-dimensional. For each $i$ take $Y_i^1$ to be the component $Y_i^\alpha$ containing the identity element of $H_n$, and take $Y_i^2$ to be the $Y_i^\alpha$ containing the transposition $\tau_i$ in $H_n$ that swaps $(i,1)$ and $(i,2)$. By construction $Y_1^{\epsilon_1}\cap\cdots\cap Y_{m(\chi)}^{\epsilon_{m(\chi)}} \ne\emptyset$ for every choice of $\epsilon_j\in\{1,2\}$, for instance this intersection contains the element of $H_n$ that is the product of those $\tau_i$ with $\epsilon_i=2$. It remains to show that for each $i$ we have $Y_i^1\ne Y_i^2$. It is enough to show that $Z_i^1\ne Z_i^2$. If $Z_i^1=Z_i^2$ (call it $Z$) then we can connect the identity to $\tau_i$ via an edge path in $Z$, and since $Z$ is combinatorially convex without loss of generality this edge path consists of a path along which $f$ strictly increases followed by a path along which $f$ strictly decreases (see \cite[Figure~3.8]{lee12} for some intuition). Since the path lies in $Z$, $\chi_i$ is non-positive on the whole path. Since $\chi_i(\id)=\chi_i(\tau_i)=0$ none of the edges in the path can be labeled by $t_i$. In particular adjacent vertices in the path must restrict to identical permutations on $\{i\}\times\N$, and hence all vertices on the path must restrict to the same permutation on $\{i\}\times\N$. Since $\id$ and $\tau_i$ do not, we have a contradiction.
\end{proof}

\begin{proof}[Proof of Theorem~\ref{thrm:BNSR_Houghton}]
By Lemma~\ref{lem:good_pieces} the Strong Nerve Lemma (Citation~\ref{cit:strong_nerve}) applies. The Strong Nerve Lemma together with Lemma~\ref{lem:houghton_nerve} says $X_{f\le 3n-3}^{0\le\chi}$ is not $(m(\chi)-1)$-acyclic. Now Lemma~\ref{lem:drop_essential} implies $[\chi]\in\Sigma^{m(\chi)}(H_n)^c$.
\end{proof}

\bibliographystyle{alpha}

\end{document}